\documentclass{article}

\usepackage{zew}
\usepackage{xcolor}
\usepackage{enumitem}
\usepackage{multirow}
\usepackage{graphicx}
\usepackage{longtable}
\usepackage{booktabs}
\usepackage{braket}
\usepackage{mathtools}

\usepackage{authblk}

\allowdisplaybreaks

\RequirePackage{natbib}
\usepackage{hyperref, cleveref}
\hypersetup{colorlinks,linkcolor={blue},citecolor={blue},urlcolor={red}} 

\title{A Quantum Dual Logarithmic Barrier Method for Linear Optimization}
\author{Zeguan Wu, Pouya Sampourmahani, \\ Mohammadhossein Mohammadisiahroudi, Tam\'as Terlaky}
\affil{Department of Industrial and Systems Engineering, Lehigh University
\\
\vspace{1mm}
[zew220, pos220, mom219, terlaky]@lehigh.edu}
\date{\today}

\begin{document}

\maketitle

\begin{abstract}
Quantum computing has the potential to speed up some optimization methods. One can use quantum computers to solve linear systems via Quantum Linear System Algorithms (QLSAs). QLSAs can be used as a subroutine for algorithms that require solving linear systems, such as the dual logarithmic barrier method (DLBM) for solving linear optimization (LO) problems. In this paper, we use a QLSA to solve the linear systems arising in each iteration of the DLBM. To use the QLSA in a hybrid setting, we read out quantum states via a tomography procedure which introduces considerable error and noise. Thus, this paper first proposes an inexact-feasible variant of DLBM for LO problems and then extends it to a quantum version. Our quantum approach has quadratic convergence toward the central path with inexact directions and we show that this method has the best-known $\mathcal{O}(\sqrt{n} \log (n \mu_0 /\zeta))$ iteration complexity, where $n$ is the number of variables, $\mu_0$ is the initial duality gap, and $\zeta$ is the desired accuracy. We further use iterative refinement to improve the time complexity dependence on accuracy. For LO problems with quadratically more constraints than variables, the quantum complexity of our method has a sublinear dependence on dimension.
\end{abstract}

\section{Introduction} \label{sec:intro}
Quantum Computing has drawn a significant amount of attention recently because of its promising speedups compared to classical computing. Such examples include integer factorization \cite{shor1994algorithms} and Quantum Linear System Algorithms (QLSAs) \cite{childs2017quantum,harrow2009quantum}. Due to the wide applications of mathematical optimization, it is natural to design and develop quantum optimization algorithms. There has been an increasing amount of research regarding developing quantum optimization algorithms for solving optimization problems. Such attempts include Quantum Approximation Optimization Algorithms (QAOA) \cite{farhi2014quantum}, Quantum Simplex Algorithm \cite{nannicini2022fast}, and Quantum Interior Point Methods (QIPMs) \cite{augustino2021inexact,augustino2021quantum,casares2020quantum,kerenidis2020quantumipm,mohammadisiahroudi2024efficient,zeguanlcqo}.

QIPMs are quantum versions of classical interior point methods (IPMs) that use QLSAs as an oracle to solve the Newton system at each iteration. Given that the solution extraction process done by quantum tomography introduces errors, inexact IPMs must be used in this context.
There are variants of interior point methods, i.e., primal, dual, and primal-dual. Previous works in the context of QIPMs include \cite{mohammadisiahroudi2023inexact} and \cite{mohammadisiahroudi2024efficient}, which proposed an inexact infeasible QIPM (II-QIPM) and an inexact feasible QIPM (IF-QIPM), respectively, for linear optimization (LO) problems. The algorithms proposed in these works are primal-dual IPM algorithms. In \cite{mohammadisiahroudi2024efficient}, the authors developed an II-QIPM that exploits QLSA to solve the normal equations system (NES) in each iteration. Their algorithm has a total complexity of 
$\tilde{\Ocal}_{n,\kappa_A,\omega}(n^4 L \|A\|^4 \omega^4 \kappa_A^2)$,  where $n$ is the number of variables, $L$ is the bit length of the input data, $\kappa_A$ is the condition number of matrix $A$, and $\omega$ is an upper bound on the norm of optimal solution. Among the systems that we can solve to obtain the Newton direction, NES is the most desirable since the coefficient matrix is positive definite. However, solving NES inexactly introduces primal infeasibility. To address this issue, Mohammadisiahroudi et al. \cite{mohammadisiahroudi2023inexact} proposed a novel system called the orthogonal subspace system (OSS) and also a modified Newton system, which runs in $\tilde{\Ocal}_{n,\kappa_A,\omega}(n^{2.5} L \|A\|^2 \omega^2 \kappa_A)$ time. In a subsequent work \cite{mohammadisiahroudi2023improvements}, they improved their complexity to $ \widetilde{\Ocal}_{ n,\kappa_{\Ahat}, \|\Ahat\|,\|\bhat\|,\mu^0} ( m\sqrt{n} L\bar{\chi}^2 \omega^2)$, where $\bar{\chi}^2$ is an upper bound on the condition number of the Newton system. Evidently, such modification comes at a computational cost. To avoid such preprocessing and modifications, we aim to use the dual logarithmic barrier method. In our algorithm, we solve NES inexactly while maintaining primal feasibility without any modifications.

There are two closely related works to ours in the context of dual logarithmic barrier methods, one in classical and one in quantum settings. First, Bellavia et al. \cite{bellavia2019inexact} proposed a dual logarithmic barrier method for solving semidefinite optimization (SDO) problems with sparse variables. As LO is a special case of SDO, by putting the elements of variables along the diagonal of the semidefinite variable, one can adopt their method for LO. Their method solves a normal equation system at each iteration. They shift the inexactness into primal feasibility, making iterations primal infeasible until optimality. Their method terminates after $\lceil 18 \sqrt{n} \log (\frac{n \mu_0}{\zeta})\rceil$ inexact Newton steps, where  $\mu_0$ is the initial duality gap, and $\zeta$ is the desired accuracy. They achieve this complexity by a tight bound on the residual. In our analysis, we show that our method requires a less strict bound on the residual.

Second, Apers and Gribling\footnote{In \cite{apers2023quantum} the authors use different notations. For consistent and unified presentation, the results of \cite{apers2023quantum} are translated to notations used in this paper.} proposed a quantum algorithm based on an IPM for so-called ``tall'' LO problems, in which $n \ggg m$, that runs in $\sqrt{n}{\rm poly}(m, \log(n), \log(1/\zeta))$, where $n$ is the number of constraints in their dual form problem \cite{apers2023quantum}. Their algorithm avoids dependence on a condition number. Under some mild assumptions and access to QRAM (a quantum-accessible data structure), their framework achieves a quantum speedup under the assumption of $n \ggg m$. To approximate the Hessian and the gradient of the barrier function in the Newton step, they proposed a quantum algorithm for the spectral approximation of the matrix $A A^T$. Further, they migrated the matrix-vector multiplication into a quantum setting using an algorithm that combines spectral approximation with quantum multivariate mean estimation algorithms. As mentioned earlier, their work focuses on tall LOs. To be more specific, they present their result to be competitive for cases where $n \geq \gamma m^{10}$ for some constant $\gamma$. %
Although they refer to some specific applications where such a setting exists, it appears that their result might not show any speedup for general cases of LO problems where $n$ and $m$ are both large and of the same order of magnitude. 
Moreover, they imposed an assumption on the size of the input data. Given their assumption, the self-dual embedding model is not applicable to finding an initial solution due to dimension restriction. Further, if one uses Khachiyan's or big-M methods, there will be a contradiction regarding the complexity. Furthermore, they did not clarify how their algorithm maintains feasibility if the gradient is approximated with some error. 

In addition, in the quantum setting, Augustino et al.~\cite{augustino2021quantum} proposed the first provably convergent quantum interior point for semidefinite optimization. If we apply this framework to LOs, we get the query complexity of $\widetilde{\Ocal}_{n, \kappa, \frac{1}{\zeta}} \left(  n^{2}  \frac{\kappa^2}{\zeta} \right)$ and $\widetilde{\Ocal}_{n, \kappa, \frac{1}{\zeta}} \left(n^{2.5}\right)$ arithmetic operations, where $\kappa$ is the condition number of Newton linear system~\cite{augustino2021quantum}. As mentioned earlier, one can adopt such a framework for solving LOs by defining diagonal matrices of input data. However, their complexity has a linear dependence on the inverse of the precision, making it exponentially costly to acquire a precise solution. Later, that dependence was improved using iterative refinement~\cite{mohammadisiahroudi2023quantum}. 

Further, among the classical works, there exist works \cite{karmarkar1984new,nesterov1994interior, roos1997theory, vaidya1989speeding} based on partial updating which calculates inexact Newton's direction using partial update of the inverse of NES. This idea gives the best complexity of $\Ocal(n^3 L)$ arithmetic operations for solving LO problems where $L$ is the input data length \cite{roos1997theory}. Researchers have used this idea besides concepts of fast matrix multiplication, spectral sparsification, inverse maintenance, and stochastic central path methods \cite{cohen2021,lee2015}. These combined approaches lead to improving the complexity of IPMs to $\Ocal(n^{\omega_0}\log(\frac{n}{\epsilon}))$, where $\omega_0 <2.3729$  is the matrix multiplication constant \citep{brand2020}. Further, \cite{van2020solving} proposed a robust primal-dual interior point method with nearly linear time for tall dense LO problems based on some techniques for efficient implementation, without using fast matrix multiplication. However, similar to \cite{apers2023quantum} there exists an assumption on the dimensions of the problem.

Due to the high cost of solving the Newton linear system within IPMs for large-scale LO problems, some papers have developed first-order methods (FOMs) for solving LO problems. FOMs exploit first-order information, i.e. gradients, to update the solution in each iteration. The cost per iteration of FOMs is significantly better than that of IPMs as they only require some of the matrix-vector products. One of the well-known FOMs is the PDLP algorithm, built upon the primal-dual hybrid gradient (PDHG) method. PDLP enhances PDHG by using presolving, preconditioning, adaptive restart, adaptive choice of step size, and primal weight. 
FOMs do not converge fast and perform well only when the input data matrix is well-conditioned and the problem has proper sharpness. PDHG with restarts has $\Ocal(\kappa \log(1/\zeta))$ iteration compelexity~\cite{applegate2023faster}. In another direction, ADMM-based IPMs (ABIPs) have been developed by solving the nonlinear log-barrier penalty function by the alternating direction method of multipliers (ADMM).  ABIPs have $\Ocal(\kappa^2/\zeta \log(1/\zeta))$ iteration complexity which is not polynomial for acquiring exact solutions~\cite{lu2024first}.

Motivated by \cite{apers2023quantum,bellavia2019inexact}, we propose a quantum version of the dual logarithmic barrier method, which is a dual IPM, for solving general LO problems. This method starts with a feasible interior point, uses Newton's method to explore the dual space, and exploits QLSAs to solve Newton's systems to obtain Newton's directions. In our convergence analysis, we prove quadratic convergence toward the central path despite having an inexact direction. Furthermore, our method has the advantage over these two works of~\cite{apers2023quantum,bellavia2019inexact} in the sense that it requires a less strict condition on the norm of residual and requires a sublinear number of queries with a much-relaxed condition on the dimension of the problem. We ultimately apply the iterative refinement technique to mitigate the dependence on precision. 

The rest of this paper is organized as follows: Sections \ref{sec:prelim} covers the preliminaries needed throughout the paper. Section \ref{sec:IF Dual} describes the analysis of our algorithm in detail and presents the complexity results. Further, the application of iterative refinement is discussed at the end of this section to improve the dependence on precision. Conclusions and future works are explained in Section \ref{sec:concl}.

\section{Preliminaries} \label{sec:prelim}
In this section, we start with some notation definitions and then briefly introduce LO problems, the dual logarithmic barrier method, and QLSA.

\subsection{Notations} \label{sec:not}
Throughout this paper $\Rmbb^n$ denotes the the set of $n$-dimensional vectors of real numbers and $\Cmbb^n$ for the set of $n$-dimensional vectors of complex numbers.
For matrix $M$, $\|M\| = \|M\|_2$ is the spectral norm, and $\|M\|_F$ is the Frobenius norm of the matrix. $\sigma_1(M)$ is the largest singular value of matrix $M$ and $\sigma_0(M)$ is the least nonzero singular value. The condition number of $M$ is defined as $\kappa(M) = \sigma_1(M)/\sigma_0(M)$.
For vectors, we use $e$ for the all-one vector. For two vectors $v_1$ and $v_2$ with the same dimension, we use $v_1 v_2$ for their entry-wise product, $v_1^p$ for the entry-wise power of $p$ for $v_1$. The $\ell_2$ norm of $v$ is denoted as $\|v\|_2$ or simply $\|v\|$ and the $\ell_1$ norm is denoted as $\|v\|_1$. Given a positive semidefinite matrix $M$, we denote $\|v\|_M = \sqrt{v^TMv}$.
For complexity, we use $\Tilde{\mathcal{O}}$ which suppresses the poly-logarithmic factors in the ``Big-O” notation. The quantities of the poly-logarithmic factors are indicated in the subscripts of $\Tilde{\mathcal{O}}$. 
\subsection{Linear Optimization Problems} \label{sec:LO}
In this paper, we consider the standard form LO problem 
 defined as follows.
\begin{definition}[Linear Optimization Problem: Standard Form]\label{def:LO}
    For vectors $b\in \mathbb{R}^m,\ c\in \mathbb{R}^n$, and matrix $A\in \mathbb{R}^{m\times n}$ with $\rank(A)=m\leq n,$ we define the primal LO problem as
    \begin{equation}\label{LO: primal}\tag{P}
    \begin{aligned} 
    \min_{x\in \Rmbb^n}\  c^Tx \ {\rm s.t.} \  Ax = b, \ x \geq 0,
    \end{aligned}
    \end{equation}
    and the dual LO problem as
    \begin{equation}\label{LO: dual}\tag{D}
        \begin{aligned}
        \max_{y\in \Rmbb^m,\ s\in \Rmbb^n} \  b^Ty\ \ {\rm s.t.\ } A^Ty + s = c,\ s \geq 0.
        \end{aligned}
    \end{equation}
\end{definition}
We define the set of primal, dual, and primal-dual feasible solutions as
\begin{equation*}
    \begin{aligned}
        \mathcal{P} &\coloneqq \left\{x\in \mathbb{R}^n:\ Ax=b,\ x\geq 0\right\},\\
        \mathcal{D} &\coloneqq \left\{(y,s)\in \mathbb{R}^m\times\mathbb{R}^n:\  A^Ty + s = c,\ s\geq 0\right\},\\
        \mathcal{PD} &\coloneqq \left\{(x,y,s)\in \mathbb{R}^n\times\mathbb{R}^m\times\mathbb{R}^n:\ Ax=b, A^Ty + s = c, (x,s)\geq 0\right\},
    \end{aligned}
\end{equation*}
and the set of interior primal-dual feasible solutions as
\begin{equation*}
    \mathcal{PD}^\circ \coloneqq \left\{(x,y,s)\in \mathbb{R}^n\times\mathbb{R}^m\times\mathbb{R}^n:\ Ax=b, \ A^Ty + s = c,\  (x,s)> 0\right\}.
\end{equation*}
Analogously, we can define the set of interior dual feasible solutions as
\begin{equation*}
    \mathcal{D}^\circ \coloneqq \left\{(y,s)\in \mathbb{R}^m\times\mathbb{R}^n:\ A^Ty + s = c,\  s> 0\right\}.
\end{equation*}

We assume that the following interior point condition (IPC) holds.
\begin{assumption}\label{assumption:ipc}
There exists a solution $(x^0,y^0,s^0)$ such that
\begin{equation*}
    Ax^0=b,\ A^Ty^0+s^0=c, \text{ and } (x^0,s^0)>0.
\end{equation*}
\end{assumption}
Note that feasible IPMs require that iterates stay in the interior of the feasible set. Thus, Assumption~\ref{assumption:ipc} is necessary. However, if it happens to fail, one can embed the LO problem into its self-dual embedding model, which is an LO problem equivalent to the original one and satisfies Assumption~\ref{assumption:ipc}. It is known that when Assumption \ref{assumption:ipc} holds, the following system has unique solution for any $\mu>0$,
\begin{equation*}
    \begin{aligned}
        Ax      &=b,\ x\geq 0\\
        A^Ty + s&=c,\ s\geq 0\\
        XSe     &=\mu e,
    \end{aligned}
\end{equation*}
where $X={\rm diag} (x)$ and $S={\rm diag} (s)$ \cite{roos1997theory}. The solution $\left(x(\mu), y(\mu), s(\mu) \right)$ is called $\mu$-center. When we only consider the dual space, we also call the point $\left(y(\mu), s(\mu) \right)$ by $\mu$-center.
%

\subsection{Dual Logarithmic Barrier Method}
\label{sec:DLBM}
In this section, we proceed with describing the dual logarithmic barrier method. The dual logarithmic barrier method focuses on the dual problem \eqref{LO: dual}. Note that any method for the dual problem can also be used for solving the primal problem too, because of the symmetry between primal and dual problems \cite{roos1997theory}.

The dual logarithmic barrier method applies Newton's method to the following dual barrier function of the dual problem \eqref{LO: dual},
\begin{equation*}
    \mathcal{L}(y) = - b^T y - \mu \sum_{i=1}^n \log \left(c_i - \left(A^T y\right)_i\right).
\end{equation*}
This is a continuous differentiable convex function and its minimum is obtained when
\begin{equation*}
    \nabla_y \mathcal{L}(y) = -b + \mu A \left(c - A^Ty\right)^{-1} =0,
\end{equation*}
where $\left(c - A^T y\right)^{-1}$ is the entry-wise inverse. Combine this with the dual equality constraint, we have the following nonlinear system
\begin{equation*}
    \begin{aligned}
        -b + \mu A S^{-1}e &=0, \\
        c - A^T y - s &=0.
    \end{aligned}
\end{equation*}
Applying Newton's method to solve the nonlinear system  gives the following Newton linear system,
\begin{equation*}
    \begin{aligned}
        \begin{bmatrix}
            0 & -\mu A S^{-2}\\
            -A^T & -I
        \end{bmatrix}
        \begin{bmatrix}
            \Delta y\\\Delta s
        \end{bmatrix}
        = -
        \begin{bmatrix}
            -b + \mu A S^{-1}e\\
            0
        \end{bmatrix}.
    \end{aligned}
\end{equation*}
Denote $r_p = b - \mu AS^{-1}e.$
The Newton linear system can be simplified into the following normal equation system (NES) as
\begin{equation}\label{eq:nes}
    \left( AS^{-2}A^T\right) \Delta y = \frac{1}{\mu} r_p. \tag{NES}
\end{equation}
It is easy to verify
\begin{align}
    \Delta y &= \frac{1}{\mu} \left( AS^{-2}A^T\right)^{-1}  r_p, \notag\\
    \Delta s &= - A^T\Delta y. \label{eq: delta s delta y}
\end{align}
The dual logarithmic barrier method starts with a strictly feasible dual solution $(y^0,\ s^0)\in \mathcal{D}^\circ$ and a $\mu^0>0$ such that $(y^0,\ s^0)$ is close to the $\mu^0$-center in the sense of the proximity measure $\delta(s^0,\mu^0)$, which is defined as
\begin{equation*}
    \begin{aligned}
        \delta(s, \mu) \coloneqq \left\| s^{-1}\Delta s \right\|_{2}.
    \end{aligned}
\end{equation*}
Then, the iterate moves along the Newton direction and finds a new iterate in the interior of $\mathcal{D}$.
Define 
\begin{equation*}
    \begin{aligned}
        x(s, \mu) = \arg\min_x \left\{ \|\mu e - sx\|\ : \ Ax = b \right\}.
    \end{aligned}
\end{equation*}
According to Theorem II.28 of \cite{roos1997theory}, we have
\begin{equation*}
    \delta(s, \mu) =\frac{1}{\mu} \min_x \left\{ \|\mu e - sx\|\ : \ Ax = b \right\}.
\end{equation*}
Algorithm~\ref{alg:DLBAfNS} is a conceptual dual logarithmic barrier method.
The following results from \cite{roos1997theory} lay the foundation for our analysis.
\begin{lemma}[Lemma II.19 in \cite{roos1997theory}] \label{lemma:feasiblity}
If the Newton step $\Delta s$ satisfies 
\begin{equation*}
    -e \leq s^{-1} \Delta s \leq e,
\end{equation*}
then $x(s, \mu)$ is primal feasible and $s^{+} = s + \Delta s$ is dual feasible. 
\end{lemma}
\begin{lemma}[Theorem II.21 in \cite{roos1997theory}]\label{lemma: quadratic from book}
    If $\delta(s, \mu) \leq 1$, then $x(s, \mu)$ is primal feasible, and $s^+ = s + \Delta s$ is dual feasible. Moreover, 
    \begin{equation*}
        \delta(s^+, \mu) \leq \delta(s, \mu)^2.
    \end{equation*}
\end{lemma}
\begin{lemma}[Theorem II.23 of \cite{roos1997theory}] \label{lemma:duality bound}
    Let $\delta \coloneqq \delta(s,\mu) \leq 1$. Then, the duality gap for the primal-dual pair $(x(s,\mu), s)$ satisfies
    \begin{equation*}
        n \mu (1-\delta) \leq s^T x(s,\mu) \leq n \mu (1+\delta).
    \end{equation*}
\end{lemma}
\begin{lemma}[Theorem II.25 of \cite{roos1997theory}]\label{lemma: original IPM complexity}
    If $\tau = 1/\sqrt{2}$ and $\theta= 1/(3\sqrt{n})$, then the dual logarithmic barrier algorithm with full Newton steps requires at most 
    \begin{equation*}
        \left\lceil 3\sqrt{n} \log\left(\frac{n\mu^0}{\epsilon}\right) \right\rceil
    \end{equation*}
    iterations. The output is a feasible primal-dual pair $(x,s)$ such that $x^T s \leq 2 \epsilon$.
\end{lemma}
In Lemma~\ref{lemma: original IPM complexity}, the complexity is also called IPM complexity because it tells how many IPM iterations are needed for an IPM to converge to the target accuracy. In Algorithm~\ref{alg:DLBAfNS}, the IPM complexity is the number of loops in the ``while" loop. In each IPM iteration, one needs to construct and solve Newton's linear system and update iterates.
\begin{algorithm}[H]
\caption{Dual Logarithmic Barrier Algorithm with full Newton steps \cite{roos1997theory}}\label{alg:DLBAfNS}
\begin{algorithmic}
\State \textbf{Input}: $\tau = 1/\sqrt{2}$, $\theta = 1/(3\sqrt{n})$, $\epsilon > 0$, $(y^0,s^0) \in \mathcal{D}^\circ$, and $\mu^0 > 0$ such that $\delta(s^0,\mu^0)\leq \tau$.
\State \textbf{begin}\\
    $s\coloneqq s^0; \mu \coloneqq \mu^0$;
    \While{$n \mu \geq (1-\theta) \epsilon$}
    \State $s \coloneqq s + \Delta s$;
    \State $\mu \coloneqq (1-\theta) \mu$;
    \EndWhile
    \State \textbf{end}
    \State \textbf{return} $(y,s)$
\end{algorithmic}
\end{algorithm}

\subsection{Quantum Linear System Algorithm} 
In this section, we start with the introduction of linear system problems (LSPs) and quantum linear system problems (QLSPs). Then, we introduce the QLSA and quantum tomography algorithm (QTA) we use in this work and discuss how to use them to solve LSPs in the IPM setting.
In this work, we assume the access to QRAM.

\subsubsection{LSP \& QLSP}
\begin{definition}[Linear System Problem (LSP)]
    The problem of finding a vector $z \in \Rmbb^n$ such that it satisfies $Mz=v$ with the coefficient matrix $M \in \Rmbb^{n \times n}$ and right-hand-side (RHS) vector $v \in \Rmbb^{n}$.
\end{definition}
LSPs could have zero, one, or infinitely many solutions. In practice especially in optimization algorithms, LSPs have a unique solution.
For general matrices, LSPs can be solved using Gaussian elimination with $\mathcal{O}(n^3)$ arithmetic operations.
When the matrices are symmetric but indefinite, LSPs can be solved exactly using Bunch–Parlett factorization with $\mathcal{O}\left(n^3\right)$ arithmetic operations \cite{bunch1971direct}.
When the matrices are symmetric and positive definite, LSPs can be solved using methods including Cholesky factorization and the Conjugate Gradient method. Among the classical methods for symmetric positive definite LSPs, the Conjugate Gradient method has the best complexity with respect to $n$ as $\mathcal{O}(nd\sqrt{\kappa}\log(1/\epsilon))$, where $d$ is the maximum number of non-zero elements in any row or column of $M$, $\kappa$ is the condition number of $M$, and $\epsilon$ is the error allowed. The price for advantage with respect to dimension is unfavorable dependence on condition number.

To define LSPs in the quantum setting, we need to introduce some notations from the quantum computing literature. The notation of ket ($\ket{\cdot}$) is used to denote quantum states. $\ket{i}$ is the $i$th computational basis quantum state, which corresponds to a classical column vector with the $i$th entry being $1$ and the rest zero.
We define LSPs in the quantum setting as follows.
\begin{definition}[Quantum Linear System Problem (QLSP)] \label{def:qlsp}
    Let $M \in \Cmbb^{n \times n}$ be a Hermitian matrix with $\|M\|_2=1$, $v \in \Cmbb^{n}$, and $z \coloneqq M^{-1} v$. We define quantum states 
    \begin{equation*}
        \ket{v} = \frac{\sum_{i=1}^{n} v_i \ket{i}}{\|\sum_{i=1}^{n} v_i \ket{i}\|}, \qquad \text{ and } \qquad \ket{z} = \frac{\sum_{i=1}^{n} z_i \ket{i}}{\|\sum_{i=1}^{n } z_i \ket{i}\|}.
    \end{equation*}
    For target precision $0 < \epsilon_{QLSP}$, the goal is to find $\ket{\ztilde}$ such that $\|\ket{\ztilde} - \ket{z}\|_2 \leq \epsilon_{QLSP}$, succeeding with probability $\Omega(1)$. 
\end{definition}

As mentioned earlier, a QIPM is an IPM that uses QLSA to solve the Newton system. To give a linear system as an input to QLSA, we need to transform the LSP into a QLSP. Moreover, the output of QLSA is a quantum state, which is not directly translatable to the classical setting. Thus, we have to use a QTA to extract the classical solution.

To transform an LSP into a QLSP, we need to determine how to encode the input data in the quantum setting. There exist two major methods for this purpose: (1) \textit{sparse-access model}, and (2) \textit{quantum operator input model}. In this paper, we choose the latter option which, according to \cite{chakraborty2018power}, is more efficient than the former when QRAM is available. 

\subsubsection{Quantum Operator Input Model}
In this model, one has access to a unitary that stores the matrix $M$,
\begin{equation*}
    U = 
    \begin{bmatrix}
        M / \alpha & . \\
        . & .
    \end{bmatrix},
\end{equation*}
where $\alpha \geq \|M\|$ is a normalization factor chosen to ensure the existence of unitary matrix $U$.
The following definition of block encoding is introduced in \cite{chakraborty2018power} with slight restatement. 
\begin{definition}[Block encoding] \label{def:blockencoding}
Let $M \in \Cmbb^{2^w \times 2^w}$ be a $w$-qubit operator. 
Then, a $(w + a)$-qubit unitary $U$ is an $(\alpha, a, \xi)$-block encoding of $M$ if 
$ U = \begin{bmatrix}
\widetilde{M} & \cdot \\
\cdot & \cdot
\end{bmatrix}$,
such that 
\begin{equation*}
    \| \alpha \widetilde{M} - M \|_2 \leq \xi.
\end{equation*}
An $(\alpha, a, \xi)$-block encoding of $M$ is said to be efficient if it can be implemented in time $T_U = \Ocal \big(\textup{poly} (w)\big).$
\end{definition}
Note that Definition~\ref{def:blockencoding} is equivalent to the following property:
\begin{equation*}
    \| M - \alpha (\bra{0}^{\otimes a} \otimes I_{2^w}) U (\ket{0}^{\otimes a} \otimes I_{2^w}) \| \leq \xi. 
\end{equation*}
In other words, we refer unitary $U$ as an $(\alpha, a, \xi)$-block encoding of $M$, where $a$ is the number
of extra qubits and $\epsilon$ is the error bound for the implementation of the block-encoding.

The following results are presented in the works of 
\cite{gilyen2019qsvt} and \citep{chakraborty2018power}.
\begin{lemma}[Lemma 50 in \cite{gilyen2019qsvt}]\label{lem:beQRAM}
Let $M \in \mathbb{C}^{2^w \times 2^w}$ be a $w$-qubit operator. If $M$ is stored in a quantum-accessible data structure, then there exist unitaries $U_R$ and $U_L$ that can be implemented in $\mathcal{O} (\textup{poly}(w \log (1/\xi)))$ time and $U_R^{\dagger} U_L$ is a $(\| M \|_F, w + 2, \xi)$-block encoding of $M$.
\end{lemma}
\begin{lemma}[Lemma 4 in \cite{chakraborty2018power}] 
  \label{prop:product}
  (Product of block-encoded matrices)
If $U_1$ is an $(\alpha_1, a_1, \xi_1)$-block-encoding of an $w$-qubit operator $M_1$, and $U_2$ is a $(\alpha_2, a_2, \xi_2)$-block-encoding of an $w$-qubit operator $M_2$, then $(I_{2^{a_2}} \otimes U_1)(I_{2^{a_1}} \otimes U_2)$ is an $(\alpha_1 \alpha_2, a_1+a_2, \alpha_1\xi_2 + \alpha_2 \xi_1)$-block-encoding of $M_1 M_2$. 
\end{lemma}
The results referred to above require QRAM, which is studied in \cite{kerenidis2016quantumrecommendation}. It allows efficient block encoding of matrices and state preparation of vectors. The physical implementation of QRAM remains an open research topic \cite{giovannetti2008quantum}. In this work, we assume the existence of QRAM.
Exploiting QRAM, one can implement an $\xi$-approximate block-encoding of $M$ with $\Ocal(\text{polylog}(\frac{n}{\xi}))$ complexity \cite{kerenidis2016quantumrecommendation}. This merely generates a polylogarithmic overhead for the total complexity. The accuracy $\xi$ here is different from the target accuracy of our QIPM. It can be shown that the accuracy $\xi$ here is a polynomial of the target accuracy of our QIPM, making the time complexity of matrices block-encoding and vectors state preparation negligible compared with the time complexity from QLSA/QTA. 

\subsubsection{QLSA \& QTA}
In our work, we use the quantum singular value transformation (QSVT) from \cite{gilyen2019qsvt} as our quantum linear system algorithm (QLSA), and QTA from \cite{van2023quantum} to solve LSPs.
When using QLSA to solve an LSP, one has to load the classical data into a quantum computer and store them properly. In this work, we assume the existence of QRAM and choose to use block encoding to preprocess the data. After properly storing the classical data in QRAM, the QLSA will output a quantum state encoding the solution of the LSP. One has to use QTA to extract a classical representation of the solution from the quantum state.
The extracted classical solution is a unit vector and it is inexact. There are algorithms that estimate the norm of the actual LSP solution but these estimations are still inexact. 
In this work, we analyze the needed accuracy for the QLSA and the QTA to ensure the expected convergence of the proposed algorithm, despite the inexactness.
We discuss the details in Section~\ref{sec: per ite complexity}.
%
\section{Inexact Feasible Dual Logarithmic Barrier Method} \label{sec:IF Dual}
In this section, first, we prove the polynomial iteration complexity of the proposed inexact feasible dual logarithmic barrier algorithm. Then,  we analyze how the inexactness of the Newton steps affects the complexity of the dual logarithmic barrier method.

\subsection{Polynomial Complexity}
In this section, we discuss how the inexactness of quantum algorithms affects the Newton step and analyze the complexity of the inexact dual logarithmic barrier method. As mentioned in the previous section, we use quantum algorithms to solve \eqref{eq:nes}, which introduces inexactness into Newton's directions. 
Let $\Delta\Tilde{y}$ be an inexact solution with inexactness $\mathcal{E}_{\Delta y}^C $, where superscription $C$ stands for classical since $\Delta \tilde y$ is a classical approximation of $\Delta y$. Let $r_{\rm NES}$ be the residual of \eqref{eq:nes}, then we have
\begin{equation*}
    \begin{aligned}
        \Delta\Tilde{y} &= \Delta y + \mathcal{E}_{\Delta y}^C,\\
        r_{\rm NES} &= \left( AS^{-2}A^T\right) (\Delta y + \mathcal{E}_{\Delta y}^C ) - \frac{1}{\mu} r_p.
   \end{aligned}
\end{equation*}
Let $\Delta \tilde{s}$ be the inexact solution computed from $\Delta\Tilde{y}$ using \eqref{eq: delta s delta y} and $\mathcal{E}_{\Delta s}^C$ be the corresponding inexactness, then
\begin{equation*}
    \mathcal{E}_{\Delta s}^C = - A^T \left( \left( AS^{-2}A^T\right)^{-1} r_{\rm NES} \right) = -A^T \mathcal{E}_{\Delta y}^C .
\end{equation*}
In the original dual logarithmic barrier method, the Newton step is accurate and thus the full Newton step can guarantee the feasibility of the new iterate, as described in Lemma \ref{lemma: quadratic from book}. In the quantum case, we only have an inexact Newton step. However, the feasibility of the new iterate after taking one full Newton step with proper conditions still can be proved. Moreover, local quadratic convergence still holds.
\begin{theorem}\label{theorem: quadratic}
If $\delta(s, \mu) \leq 0.5$, then $x(s, \mu)$ is primal feasible. Furthermore, if $\left\|s^{-1} \mathcal{E}_{\Delta s}^C\right\|_{2}\leq \frac{1}{3} \delta(s, \mu)^2$, then $s^+ = s + \Delta s + \mathcal{E}_{\Delta s}^C$ is dual feasible. Moreover, 
\begin{equation*}
    \delta(s^+, \mu) \leq 1.5 \delta(s, \mu)^2.
\end{equation*}
\end{theorem}
\begin{proof}
    Following Lemma \ref{lemma: quadratic from book}, $x(s, \mu)$ is primal feasible because $\delta(s, \mu)\leq 0.5<1$.
    Then, we prove $s^+ = s + \Delta s + \mathcal{E}_{\Delta s}^C$ is dual feasible. Notice that
    \begin{equation*}
        \begin{aligned}
            s^{-1}s^+ &= e + s^{-1}(\Delta s + \mathcal{E}_{\Delta s}^C).
        \end{aligned}
    \end{equation*}
    For any $i\in \{1,\dots, n\}$, since $\delta(s, \mu)\leq 0.5$, we have
    \begin{equation*}
        \left(s^{-1} \Delta s \right)_i \geq -\delta(s, \mu) \geq -0.5.
    \end{equation*}
    Similarly, since $\left\|s^{-1} \mathcal{E}_{\Delta s}^C\right\|_{2}\leq \frac{1}{3} \delta(s, \mu)^2$, we have
    \begin{equation*}
        \left(s^{-1} \mathcal{E}_{\Delta s}^C\right)_{i}\geq -\frac{1}{3} \delta(s, \mu)^2 \geq -\frac{1}{3}\times (0.5)^2.
    \end{equation*}
    It follows that
    \begin{equation*}
        s^{-1}s^+ \geq e - (0.5)e - \frac{1}{3}\times(0.5)^2 e \simeq 0.417e > 0,
    \end{equation*}
    and thus $s^+$ is dual feasible.
    Finally, we prove the local quadratic convergence of the iterative sequence to the central path. Following the proof of Theorem II.20 in \cite{roos1997theory}, we have
    \begin{equation*}
        \begin{aligned}
            \delta(s^+, \mu) \leq \frac{1}{\mu}\left\|\mu e - s^+ x(s, \mu)  \right\|_{2},
        \end{aligned}
    \end{equation*}
    and 
    \begin{equation*}
        \mu s^{-1} \Delta s = \mu e - s x(s, \mu).
    \end{equation*}
    Combining with the definition of $s^+$, we have
    \begin{equation*}
        \begin{aligned}
            \mu e - s^+ x(s, \mu) &= \mu e - (s + \Delta s + \mathcal{E}_{\Delta s}^C)x(s, \mu)\\
            &= \mu\left( s^{-1} \Delta s \right)^2 - \mathcal{E}_{\Delta s}^C x(s, \mu)\\            &= \mu\left( s^{-1} \Delta s \right)^2 + \left(\mu s^{-1}\Delta s - \mu e  \right)s^{-1}\mathcal{E}_{\Delta s}^C .
        \end{aligned}
    \end{equation*}
    By the definition of $\delta(s^+, \mu)$, we have
    \begin{equation*}
        \begin{aligned}
            \delta(s^+, \mu) &\leq \frac{1}{\mu}\left\|\mu e - s^+ x(s, \mu)  \right\|_{2}\\
            &\leq \delta(s, \mu)^2 + \left\|  \left(s^{-1}\Delta s -  e  \right)s^{-1}\mathcal{E}_{\Delta s}^C \right\|_{2},
        \end{aligned}
    \end{equation*}
    where the second inequality follows the triangular inequality.
    For the second term, we further relax it by
    \begin{equation*}
        \begin{aligned}
            \left\| \left(s^{-1}\Delta s -  e  \right)s^{-1}\mathcal{E}_{\Delta s}^C \right\|_{2} 
            &\leq \left\| \left(s^{-1}\Delta s -  e  \right)\right\|_{\infty} \left\|s^{-1}\mathcal{E}_{\Delta s}^C \right\|_{2}.
        \end{aligned}
    \end{equation*}
    Since $\delta(s, \mu)\leq \frac{1}{2}$, we have
    \begin{equation*}
        \begin{aligned}
            \left\|  \left(s^{-1}\Delta s -  e  \right)\right\|_{\infty} \leq 1.5.
        \end{aligned}
    \end{equation*}
    Thus,
    \begin{equation*}
        \delta(s^+, \mu) \leq \delta(s, \mu)^2 + 1.5 \left\|s^{-1}\mathcal{E}_{\Delta s}^C \right\|_{2} \leq 1.5\delta(s, \mu)^2,
    \end{equation*}
    which completes the proof.
\end{proof}

\begin{lemma}
    If \textcolor{black}{$\left\|s^{-1} \mathcal{E}_{\Delta s}^C \right\|_{2}\leq 0.1\delta(s, \mu)$}, and $\delta(s, \mu) \leq 0.5$, let $s^+ = s + \Delta s + \mathcal{E}_{\Delta s}^C $ and $\mu^+ = (1-\theta)\mu$. Then we have 
    \begin{equation*}
        \delta(s^+, \mu^+)^2 \leq  (1+\rho^2)\delta(s, \mu)^4 + 0.06(1+\rho^2) + \rho^2n,
    \end{equation*}
    where $\rho= \frac{\theta}{1-\theta}.$
\end{lemma}
\begin{proof}
    Following the proof of Lemma II.26 of \cite{roos1997theory}, we have
    \begin{equation*}
        \delta(s^+, \mu^+)^2 \leq \|h\|^2 - 2\rho h^T(e-h) + \rho^2 \|e-h\|^2,
    \end{equation*}
    where
    \begin{equation*}
        \begin{aligned}
            h &= \left( s^{-1} \Delta s \right)^2 + \left( s^{-1}\Delta s -  e  \right)s^{-1}\mathcal{E}_{\Delta s}^C .
        \end{aligned}
    \end{equation*}
    Since $\left\|s^{-1} \mathcal{E}_{\Delta s}^C \right\|_{2}\leq \delta(s, \mu)\leq 0.5$, we have
    \begin{equation*}
        \begin{aligned}
            h &\leq 0.5^2 e + (-0.5 - 1)(-0.5) e = e.
        \end{aligned}
    \end{equation*}
    For
    \begin{equation*}
        \begin{aligned}
            h^T(e-h) &= \left( e - \left( s^{-1} \Delta s \right)^2 - \left( s^{-1}\Delta s -  e  \right)s^{-1}\mathcal{E}_{\Delta s}^C \right)^T \left( \left( s^{-1} \Delta s \right)^2+ \left( s^{-1}\Delta s -  e  \right) s^{-1}\mathcal{E}_{\Delta s}^C \right),
        \end{aligned}
    \end{equation*}
notice that
\begin{equation}
    \begin{aligned}
        &\left( e - \left( s^{-1} \Delta s \right)^2 - \left( s^{-1}\Delta s -  e  \right)s^{-1}\mathcal{E}_{\Delta s}^C \right)^T \left( s^{-1} \Delta s \right)^2 \\
        \geq &\left( 1 - \left( 0.5\right)^2 - (-0.5-1)(-0.1\times 0.5)\right)e^T \left( 0.5 \right)^2e\\
        \geq &0.16875n.
    \end{aligned}
\end{equation}
Also, if $\left( \left( s^{-1}\Delta s -  e  \right) s^{-1}\mathcal{E}_{\Delta s}^C \right)_i\geq 0$, then we have
\begin{equation*}
    \begin{aligned}
        &\left( e - \left( s^{-1} \Delta s \right)^2 - \left( s^{-1}\Delta s -  e  \right)s^{-1}\mathcal{E}_{\Delta s}^C \right)_i \left( \left( s^{-1}\Delta s -  e  \right) s^{-1}\mathcal{E}_{\Delta s}^C \right)_i\\
        \geq & \left( 1 - 0.5^2 - \left( -0.5 -  1  \right)(-0.1\times0.5)\right) \left( \left( s^{-1}\Delta s -  e  \right) s^{-1}\mathcal{E}_{\Delta s}^C \right)_i\\
        \geq &\  0.
    \end{aligned}
\end{equation*}
If $\left( \left( s^{-1}\Delta s -  e  \right) s^{-1}\mathcal{E}_{\Delta s}^C \right)_i < 0$, then we have
\begin{equation*}
    \begin{aligned}
        &\left( e - \left( s^{-1} \Delta s \right)^2 - \left( s^{-1}\Delta s -  e  \right)s^{-1}\mathcal{E}_{\Delta s}^C \right)_i \left( \left( s^{-1}\Delta s -  e  \right) s^{-1}\mathcal{E}_{\Delta s}^C \right)_i\\
        \geq & ~(1 - 0 - (-0.5-1)\times(0.1\times0.5))(-0.5-1)(0.1\times0.5)\\
        \geq & -0.080625,
    \end{aligned}
\end{equation*}
which gives
\begin{equation*}
    \begin{aligned}
    &\left( e - \left( s^{-1} \Delta s \right)^2 - \left( s^{-1}\Delta s -  e  \right)s^{-1}\mathcal{E}_{\Delta s}^C \right)^T \left( \left( s^{-1}\Delta s -  e  \right) s^{-1}\mathcal{E}_{\Delta s}^C \right)\\
            &\geq -0.080625n.
    \end{aligned}
\end{equation*}
It follows that
\begin{equation*}
    \begin{aligned}
        h^T(e-h)\geq 0.
    \end{aligned}
\end{equation*}
    Moreover, we have
    \begin{equation*}
        \begin{aligned}
        \|h\|^2 & = \left\| \left( s^{-1} \Delta s \right)\right\|^4 + \left\| 2 \left(s^{-1} \Delta s\right)^2 \left( s^{-1}\Delta s -  e  \right)s^{-1}\mathcal{E}_{\Delta s}^C  + \left( s^{-1}\Delta s -  e  \right)^2s^{-2}(\mathcal{E}_{\Delta s}^C )^2  \right\|  \\
        &\leq \left\| \left( s^{-1} \Delta s \right)\right\|^4 + \left\| 2 \left(s^{-1} \Delta s\right)^2 \left( s^{-1}\Delta s -  e  \right)s^{-1}\mathcal{E}_{\Delta s}^C  + \left( s^{-1}\Delta s -  e  \right)^2s^{-2}(\mathcal{E}_{\Delta s}^C )^2  \right\|_1 \\
        &\leq \left\| \left( s^{-1} \Delta s \right)\right\|^4 + \left\| 2 \left(s^{-1} \Delta s\right)^2  + \left( s^{-1}\Delta s -  e  \right) s^{-1}\mathcal{E}_{\Delta s}^C   \right\| \left\| \left( s^{-1}\Delta s -  e  \right)s^{-1}\mathcal{E}_{\Delta s}^C \right\|,
    \end{aligned}
    \end{equation*}
where the second inequality follows the H\"older's inequality. Given that
\begin{equation*}
    \begin{aligned}
        \left\|\left( s^{-1}\Delta s -  e  \right) s^{-1}\mathcal{E}_{\Delta s}^C   \right\| \leq& \left\| s^{-1}\Delta s s^{-1}\mathcal{E}_{\Delta s}^C  \right\| + \left\|  e s^{-1}\mathcal{E}_{\Delta s}^C   \right\|\\
        \leq & \left\| s^{-1}\Delta s  \right\|_{\infty} \left\| s^{-1}\mathcal{E}_{\Delta s}^C  \right\| + \left\| e\right\|_{\infty} \left\| s^{-1}\mathcal{E}_{\Delta s}^C  \right\|\\
        \leq & 2\left\| s^{-1}\mathcal{E}_{\Delta s}^C  \right\|,
    \end{aligned}
\end{equation*}
and
\begin{equation*}
    \begin{aligned}
        \left\| 2 \left(s^{-1} \Delta s\right)^2  + \left( s^{-1}\Delta s -  e  \right) s^{-1}\mathcal{E}_{\Delta s}^C   \right\|
        \leq& \left\| 2 \left(s^{-1} \Delta s\right)^2  \right\| +  \left\|\left( s^{-1}\Delta s -  e  \right) s^{-1}\mathcal{E}_{\Delta s}^C   \right\|\\
        \leq &\left\| 2 \left(s^{-1} \Delta s\right)^2  \right\| + 2\left\|  s^{-1}\mathcal{E}_{\Delta s}^C   \right\|\\
        \leq& 2\left\| \left(s^{-1} \Delta s\right)  \right\|^2 + 2\left\|  s^{-1}\mathcal{E}_{\Delta s}^C   \right\|,
    \end{aligned}
\end{equation*}
we have
\begin{align*}
    \|h\|^2 &\leq \left\| \left( s^{-1} \Delta s \right)\right\|^4 + \left( 2\left\| \left(s^{-1} \Delta s\right)  \right\|^2  +  2\left\| s^{-1}\mathcal{E}_{\Delta s}^C   \right\| \right) 2\left\|s^{-1}\mathcal{E}_{\Delta s}^C \right\| \\
    &\leq \delta(s, \mu)^4 + 0.06. 
\end{align*}
    Next, for $\|e - h\|_2^2$, we have
    \begin{equation*}
        \begin{aligned}
            \|e - h\|_2^2 &\leq \|e\|_2^2 + \|h\|_2^2\\
            &\leq n + \|h\|_2^2.
        \end{aligned}
    \end{equation*}
    Finally, for the proximity measure, we obtain
    \begin{equation*}
        \begin{aligned}
            \delta(s^+, \mu^+)^2 &\leq \|h\|^2 - 2\rho h^T(e-h) + \rho^2 \|e-h\|^2\\
            &\leq (1+\rho^2)\|h\|^2 + \rho^2n\\
            &\leq (1+\rho^2)[0.06+ \delta(s, \mu)^4] + \rho^2n,
        \end{aligned}
    \end{equation*}
    which shows the claimed bound.
\end{proof}

\begin{theorem}\label{theorem: polynomial complexity}
    If $\delta(s,\mu) \leq 0.5$, $\left\|s^{-1} \mathcal{E}_{\Delta s}^C \right\|_{2}\leq 0.1\delta(s, \mu)$, and $\theta= 1/(4\sqrt{n})$, then the inexact dual logarithmic barrier algorithm with full Newton steps requires at most 
    \begin{equation*}
        \left\lceil 4\sqrt{n} \log\left(\frac{n\mu^0}{\epsilon}\right) \right\rceil
    \end{equation*}
    iterations. The output is a primal-dual pair $(x,s)$ such that $x^T s \leq 2 \epsilon$.
\end{theorem}
\begin{proof}
    Let $\theta = \frac{1}{4\sqrt{n}}$, then we have
    \begin{equation*}
        \begin{aligned}
            \rho = \frac{\theta}{1-\theta} \leq \frac{1/4}{1 - 1/4}\frac{1}{\sqrt{n}} =\frac{1}{3\sqrt{n}}.
        \end{aligned}
    \end{equation*}
    With $\delta(s,\mu)\leq 0.5$, we have
    \begin{equation*}
        \begin{aligned}
            0.06+ \delta(s, \mu)^4 &\leq 0.1225,
        \end{aligned}
    \end{equation*}
    which gives
    \begin{equation*}
        \begin{aligned}
            \delta(s^+, \mu^+)^2 &\leq 0.1225 + \rho^2(0.1225 +n)\\
            &\leq 0.1225 + \frac{0.1225 +n}{9n}\\
            &\leq 0.25.
        \end{aligned}
    \end{equation*}
    Hence, after each iteration of the proposed inexact feasible dual logarithmic barrier algorithm, the property $\delta(s,\mu)\leq0.5$ holds.
    Let $n\mu_0(1-\theta)^K \geq (1-\theta)\epsilon$, one gets the iteration bound $K\leq \left\lceil 4\sqrt{n} \log\left({n\mu^0}/{\epsilon}\right) \right\rceil$. Finally, following Lemma~\ref{lemma:duality bound}, we obtain
    \begin{equation*}
        s^Tx(s,\mu) \leq n\mu (1+\delta(s,\mu))\leq 2n\mu\leq 2\epsilon.
    \end{equation*}
\end{proof}
%
Now, we need to determine what accuracy is required to satisfy the aforementioned conditions.
Since we assume that we start with the two conditions, i.e., $\delta(s,\mu)\leq 0.5$ and $\left\|s^{-1} \mathcal{E}_{\Delta s}^C \right\|_{2}\leq 0.1\delta(s, \mu)$, satisfied, the condition $\delta(s, \mu) \leq 0.5$ automatically holds in the next iteration as proved in the proof of Theorem~\ref{theorem: polynomial complexity}. The remaining question is how to maintain $\left\|s^{-1} \mathcal{E}_{\Delta s}^C \right\|_{2}\leq  0.1\delta(s, \mu)$.
In the next section, we answer this question and analyze the per-iteration complexity.
\subsection{Per-iteration Complexity} \label{sec: per ite complexity}
In this section, we analyze the cost of each IPM iteration. Specifically, we analyze the number of queries to QRAM and the number of classical arithmetic operations. We start by analyzing the rescaling factor for the normalized solution from QLSA+QTA and then analyze the quantum accuracy needed for QLSA+QTA.
The analysis in this section holds for any iteration. For simplicity of notation, we omit the superscripts for iteration number.
Recall that the convergence conditions in Theorem~\ref{theorem: polynomial complexity} require $\|s^{-1} \mathcal{E}_{\Delta s}^C \|_2$ be bounded, so we try to find the proper rescaling factor to minimize the term. By the definition of $\mathcal{E}_{\Delta s}^C$, the term can be represented as
\begin{equation*}
    \begin{aligned}
        \|s^{-1} \mathcal{E}_{\Delta s}^C \|_2 &= \| S^{-1} A^T (\Delta y - \Delta \tilde{y}) \|_2 = \| S^{-1} A^T (\Delta y - \lambda \Delta\bar y /\|\Delta\bar y \|_2) \|_2,
    \end{aligned}
\end{equation*}
where $\lambda$ is the rescaling factor and $\Delta \bar y $ is a unit vector obtained from the quantum subroutine. Recall that $\Delta y$ is the exact solution (not necessarily a unit vector) and $\Delta \tilde{y}$ is an approximate solution rescaled from $\Delta \bar y $ with factor, we have the following relationships
\begin{equation*}
\begin{aligned}
\Delta \tilde{y} &= \lambda^* \Delta \bar{y},\\
\Delta y &= (\Delta\bar{y} - \mathcal{E}^Q_{\Delta y})\|\Delta y\|_2,\\
\Delta\bar{y} &= \Delta y/\|\Delta y\|_2 + \mathcal{E}_{\Delta y}^Q,
\end{aligned}
\end{equation*}
where $\lambda^*$ is the factor we use to rescale $\Delta \bar{y}$ from the quantum subroutine and is defined as the minimizer of $\|s^{-1}\mathcal{E}_{\Delta s}^C\|_2$ as a function of $\lambda$. It is obvious that
\begin{equation}\label{eq: lambda ast}
    \begin{aligned}
        \lambda^* = \frac{ (S^{-1}A^T\Delta y)^T ( S^{-1}A^T\Delta\bar  y) }{( S^{-1}A^T\Delta\bar  y)^T S^{-1}A^T\Delta\bar  y} = \frac{(A S^{-2} A^T \Delta y)^T \Delta \bar{y}}{\|S^{-1}A^T \Delta \bar{y}\|_2^2} = \frac{r_p^T\Delta \bar{y}}{\mu \|S^{-1}A^T \Delta \bar{y}\|_2^2}.
    \end{aligned}
\end{equation}
Notice that
\begin{equation*}
    \begin{aligned}
        \|s^{-1} \mathcal{E}_{\Delta s}^C \|_2 &= \| S^{-1} A^T \Delta y - 
        \frac{ (S^{-1}A^T\Delta y)^T ( S^{-1}A^T\Delta\bar  y) }{( S^{-1}A^T\Delta\bar  y)^T S^{-1}A^T\Delta\bar  y}
        S^{-1} A^T\Delta\bar  y  \|_2\\
        &= \|S^{-1}A^T\Delta y\|_2 \left\| \frac{ S^{-1}A^T\Delta y}{\| S^{-1}A^T\Delta y\|_2} - \frac{ (S^{-1}A^T\Delta y )^T  S^{-1}A^T\Delta\bar  y}{\| S^{-1}A^T\Delta y\|_2 \| S^{-1}A^T\Delta\bar  y\|_2} \frac{ S^{-1}A^T\Delta\bar  y}{\| S^{-1}A^T\Delta\bar  y\|_2} \right\|_2\\
        &= \delta(s,\mu) \sin<S^{-1}A^T\Delta y, S^{-1}A^T\Delta\bar  y>,
    \end{aligned}
\end{equation*}
the convergence condition $\|s^{-1}\mathcal{E}_{\Delta s}^C \|_2 \leq 0.1 \delta(s,\mu)$ can be guaranteed by
\begin{equation*}
    \begin{aligned}
        \sin \langle S^{-1}A^T\Delta y, S^{-1}A^T\Delta\bar  y \rangle \leq 0.1.
    \end{aligned}
\end{equation*}
We propose the following proposition, and we use the rest of the section to prove it.
\begin{proposition}\label{poroposition:accuracy}
    If $\|\mathcal{E}_{\Delta y}^Q\|_2\leq \frac{0.005}{1.995}\frac{1}{\sqrt{\kappa(AS^{-2}A^T)}}$, then $\sin \langle S^{-1}A^T\Delta y, S^{-1}A^T\Delta\bar  y \rangle \leq 0.1.$
\end{proposition}
We start the proof with the following lemma.
\begin{lemma}\label{lemma: a lemma}
    Let $u$ and $v$ be unit vectors in $\mathbb{R}^n$ and $M\in\mathbb{R}^{n\times n}$ be positive definite. Let $0\leq \gamma_1\leq \gamma_2/\sqrt{\kappa(M)}$ with $\gamma_2\in [0,1]$. Then
    \begin{equation*}
        u^T M (\gamma_1 v) + u^TMu \geq 0.
    \end{equation*}
\end{lemma}
\begin{proof}
    As matrix $M$ is positive definite we factorize it as $M = M_L^TM_L$. Let $\tilde u = M_Lu$ and $\tilde v = M_Lv$.
    Then $u^T M (\gamma_1v) + u^TMu = \tilde u^T \tilde v \gamma_1 + \tilde u^T \tilde u.$
    Notice that
    \begin{equation*}
        \begin{aligned}
            \|\tilde v \gamma_1\|_2 = \|\tilde v\|_2 \gamma_1 \leq \sqrt{\sigma_1(M)}\gamma_1,
        \end{aligned}
    \end{equation*}
    where $\sigma_1(M)$  is the maximum singular value of $M$.
    It is obvious that
    \begin{equation*}
        \begin{aligned}
            \tilde u^T \tilde v \gamma_1 \geq \tilde u^T (\frac{-\tilde u}{\|\tilde u\|_2}) \sqrt{\sigma_1(M)}\gamma_1.
        \end{aligned}
    \end{equation*}
    It follows that
    \begin{equation*}
        \begin{aligned}
            u^T M (\gamma_1v) + u^TMu &\geq -\|\tilde u\|_2 \sqrt{\sigma_1(M)}\gamma_1 + \|\tilde u\|_2^2\\
            &= (\|\tilde u\|_2 - \frac{\sqrt{\sigma_1(M)}\gamma_1}{2})^2 - \frac{\sigma_1(M)\gamma_1^2}{4}.
        \end{aligned}
    \end{equation*}
    Notice that
    \begin{equation*}
        \begin{aligned}
            \|\tilde u\|_2 - \frac{\sqrt{\sigma_1(M)}\gamma_1}{2}&\geq \|\tilde u\|_2 - \frac{\sqrt{\sigma_1(M)}\gamma_2}{2\sqrt{\kappa(M)}}\\
            &=\|\tilde u\|_2 - \frac{\sqrt{\sigma_0(M)}\gamma_2}{2}\\
            &\geq \sqrt{\sigma_0(M)} - \frac{\sqrt{\sigma_0(M)}\gamma_2}{2}\\
            &= (1-\frac{\gamma_2}{2})\sqrt{\sigma_0(M)}\\
            &\geq 0,
        \end{aligned}
    \end{equation*}
    where $\sigma_0(M)$ is the minimum singular value of $M$.
    So we have
    \begin{equation*}
        \begin{aligned}
            u^TM(\gamma_1v) + u^TMu &\geq (1-\frac{\gamma_2}{2})^2\sigma_0(M) - \sigma_1(M)\gamma_1^2/4\\
            &\geq (1-\frac{\gamma_2}{2})^2\sigma_0(M) - \frac{\gamma_2^2}{4}\sigma_0(M)\\
            &= (1-\gamma_2)\sigma_0(M)\\
            &\geq 0.
        \end{aligned}
    \end{equation*}
The proof is complete.
\end{proof}
Denote the angle between $S^{-1}A^T \Delta y$ and $S^{-1}A^T\Delta \bar{y}$ by $\psi$.
By the definition of cosine, we have 
\begin{equation*}
    \begin{aligned}
        \cos \psi &= \frac{\Delta y^T A S^{-2}A^T\Delta \bar  y}{\|S^{-1} A^T\Delta y\|_2 \|S^{-1}  A^T\Delta \bar  y\|_2}\\
        &= \frac{\Delta y^T A S^{-2}A^T(\Delta y + \|\Delta y\|_2 \mathcal{E}_{\Delta y}^Q))}{\|S^{-1} A^T\Delta y\|_2 \|S^{-1}  A^T(\Delta y + \|\Delta y\|_2 \mathcal{E}_{\Delta y}^Q)\|_2}\\
        &= \frac{\|\Delta y\|_2\Delta y^T A S^{-2}A^T \mathcal{E}_{\Delta y}^Q + \Delta y^T A S^{-2}A^T\Delta y}{\|\Delta y\|_2\|S^{-1} A^T\Delta y\|_2 \|S^{-1}  A^T \mathcal{E}_{\Delta y}^Q\|_2 + \Delta y^T A S^{-2}A^T\Delta y}.
    \end{aligned}
\end{equation*}
The following inequality holds because $\|\mathcal{E}_{\Delta y}^Q\|_2\leq \frac{0.005}{1.995}\frac{1}{\sqrt{\kappa(AS^{-2}A^T)}}$ and Lemma~\ref{lemma: a lemma},
\begin{equation*}
    \begin{aligned}
        \|\Delta y\|_2\Delta y^T A S^{-2}A^T \mathcal{E}_{\Delta y}^Q + \Delta y^T A S^{-2}A^T\Delta y &= \|\Delta y\|_2^2 \left( \frac{\Delta y^T}{\|\Delta y\|_2}AS^{-2}A^T \mathcal{E}_{\Delta y}^Q + \frac{\Delta y^T}{\|\Delta y\|_2} AS^{-2}A^T \frac{\Delta y}{\|\Delta y\|_2}  \right)\geq 0.
    \end{aligned}
\end{equation*}
It is obvious that $\cos \langle S^{-1}A^T\Delta y, S^{-1}A^T\Delta\bar  y \rangle \geq 0.995$ implies $\sin \langle S^{-1}A^T\Delta y, S^{-1}A^T\Delta\bar  y \rangle \leq 0.1$.
To get $\cos \langle S^{-1}A^T\Delta y, S^{-1}A^T\Delta \bar  y \rangle \geq 0.995,$
we enforce \begin{equation*}
    \begin{aligned}
        0.005\|\Delta y\|_{AS^{-2}A^T}^2 &\geq 1.995 \|\Delta y\|_2 \|S^{-1} A^T \Delta y\|_2\|S^{-1} A^T \mathcal{E}_{\Delta y}^Q\|_2,
    \end{aligned}
\end{equation*}
which gives
\begin{equation*} 
    \begin{aligned}
        \cos \psi&= \frac{\|\Delta y\|_2\Delta y^T A S^{-2}A^T \mathcal{E}_{\Delta y}^Q + \|\Delta y\|_{AS^{-2}A^T}^2}{\|\Delta y\|_2\|S^{-1} A^T\Delta y\|_2 \|S^{-1}  A^T \mathcal{E}_{\Delta y}^Q\|_2 + \|\Delta y\|_{AS^{-2}A^T}^2}\\
        &\geq \frac{-\|\Delta y\|_2\|S^{-1} A^T\Delta y\|_2 \|S^{-1}  A^T \mathcal{E}_{\Delta y}^Q\|_2 + \|\Delta y\|_{AS^{-2}A^T}^2}{\|\Delta y\|_2\|S^{-1} A^T\Delta y\|_2 \|S^{-1}  A^T \mathcal{E}_{\Delta y}^Q\|_2 + \|\Delta y\|_{AS^{-2}A^T}^2}\\
        &\geq 0.995.
    \end{aligned}
\end{equation*}
Notice that $\|\Delta y\|_{AS^{-2}A^T} = 0$ happens only when $r_p=0$, in which case $\Delta y = 0$ and there is no need to use a linear system solver. Thus we only need to analyze the case where $\|\Delta y\|_{AS^{-2}A^T} > 0$. Then we only need to enforce
\begin{equation*}
    \begin{aligned}
        0.005\|\Delta y \|_{AS^{-2}A^T} \geq 1.995 \|\Delta y\|_2 \|S^{-1} A^T \mathcal{E}_{\Delta y}^Q\|_2,
    \end{aligned}
\end{equation*}
which can be guaranteed if
\begin{equation*}
    \begin{aligned}
        \|\mathcal{E}_{\Delta y}^Q\|_2 &\leq \frac{0.005}{1.995} \frac{\|S^{-1}A^T\Delta y\|_2}{\|S^{-1}A^T\|_2 \|\Delta y\|_2}\\
        &= \frac{0.005}{1.995} \sqrt{\frac{\Delta y^T A S^{-1} S^{-1} A^T \Delta y}{\|A S^{-2} A^T\|_2 \|\Delta y\|_2^2} }.
    \end{aligned}
\end{equation*}
By the definition of condition number, we have
\begin{equation*}
    \begin{aligned}
    \frac{0.005}{1.995} \frac{1}{\sqrt{\kappa({AS^{-2}A^T}))}} \leq \frac{0.005}{1.995} \sqrt{\frac{\Delta y^T A S^{-1} S^{-1} A^T \Delta y}{\|A S^{-2} A^T\|_2 \|\Delta y\|_2^2} }.
    \end{aligned}
\end{equation*}
So, the following condition guarantees the convergence condition $\|s^{-1}\mathcal{E}_{\Delta s}^C \|_2 \leq 0.1 \delta(s,\mu)$  in Theorem~\ref{theorem: polynomial complexity}:
\begin{equation*}
    \|\mathcal{E}_{\Delta y}^Q\|_2 \leq \frac{0.005}{1.995} \frac{1}{\sqrt{\kappa({AS^{-2}A^T})}}.
\end{equation*}
From the quantum accuracy, we are ready to discuss the complexity of solving Newton system using QLSA+QTA. Finally, to derive the main result, we also need the following lemma, which can be proved using singular value decomposition.
\begin{lemma}\label{lemma:cond}
    For any full row rank matrix $Q \in \Rmbb^{m \times n}$ and any symmetric positive definite matrix $\Psi \in \Rmbb^{n \times n}$, their condition number satisfies 
    \begin{equation*}
        \kappa(Q \Psi Q^T) = \Ocal(\kappa(\Psi)\kappa_Q^2).
    \end{equation*}
\end{lemma}
To apply Lemma \ref{lemma:cond} to $AS^{-2}A^T$, it is easy to see that $Q = A$, recalling the assumption that $A$ is full row rank. Further, $\Psi = S^{-2}$, is a symmetric positive definite matrix. Hence,
\begin{equation*}
    \kappa(AS^{-2} A^T) = \Ocal(\kappa(S^{-2})\kappa_A^2).
\end{equation*}
Further, let $ \omega = \max_{k}\|s^{(k)}\|_{\infty}$, where $s^{(k)}$ is the dual slack variable in the $k^{\rm th}$ iteration. According to \cite{roos1997theory}, we have 
$\kappa(S^{-2}) = \Ocal({\omega^4}/{\mu^2})$, which gives
\begin{equation*}
    \left\|\mathcal{E}_{\rm \Delta y}^Q\right\|_2 = \mathcal{O}\left( \frac{\mu}{\kappa_A \omega^2} \right).
\end{equation*}
\begin{theorem}\label{Subroutine_Complexity}
    Given the linear system \eqref{eq:nes}, using the QSVT from \cite{gilyen2019thesis} and the QTA from \cite{van2023quantum}, an $\epsilon$-approximate solution $\widetilde{\Delta y}$ can be obtained in at most 
    \begin{equation*}
    \widetilde{\Ocal}_{n,\|AS^{-1}\|_F,\kappa(AS^{-1}), \frac{1}{\epsilon}}(m\kappa^2(AS^{-1}))
    \end{equation*}
    queries to QRAM.
\end{theorem}
\begin{proof}
According to \cite{mohammadisiahroudi2024quantum}, since the normal equation can be rewritten as 
\begin{equation*}
    (AS^{-1})(AS^{-1})^T \Delta y = (AS^{-1})\frac{1}{\mu}\left(Sx_0 - \mu e \right),
\end{equation*}
the normal equation systems can be solved using quantum singular value transformation (QSVT)~\cite{gilyen2019thesis,gilyen2019qsvt} with the complexity of $\widetilde{\Ocal}_{n,\kappa(AS^{-1}),\frac{1}{\epsilon} }(\kappa(AS^{-1}) \|AS^{-1}\|_F)$. 
Furthermore, iterative refinement can mitigate the linear dependence on $\|AS^{-1}\|_F$ into polylogarithmic dependence~\cite{ mohammadisiahroudi2024quantum}.
It is worth mentioning that the iterative refinement algorithm in \cite{mohammadisiahroudi2024quantum} does not necessarily output a normalized solution. But we can use the method of \cite{mohammadisiahroudi2024quantum} to compute a solution with accuracy ${\rm poly}(\epsilon, 1/\|AS^{-1}\|_F, 1/\kappa(AS^{-1}))$ and normalize it, which only contributes ${\rm polylog}(\|AS^{-1}\|_F, \kappa(AS^{-1}))$ overhead and gives an $\epsilon$-accurate normalized solution.
Next, to extract the classical solution, we use the quantum tomography algorithm (QTA) of \cite{van2023quantum} with accuracy $\epsilon_{\Delta y}^Q$, where $\epsilon_{\Delta y}^Q = \left\|\mathcal{E}_{\Delta y}^Q\right\|_2$. Thus, solving the \eqref{eq:nes} with QSVT and QTA to the desired accuracy requires 
\begin{equation*}
    \widetilde{\Ocal}_{n,\|AS^{-1}\|_F, \kappa(AS^{-1}),\frac{1}{\epsilon_{\Delta y}^Q} } \left( \frac{m \kappa(AS^{-1})}{\epsilon_{\Delta y}^Q} \right)
\end{equation*}
queries to QRAM. Given that $\epsilon_{\Delta y}^Q = \left\|\mathcal{E}_{\rm \Delta y}^Q\right\|_2  = 1/\sqrt{\kappa(AS^{-2}A^T)} = \kappa^{-1}(AS^{-1})$, our quantum complexity requires
\begin{equation*}
    \widetilde{\Ocal}_{n,\|AS^{-1}\|_F,\kappa(AS^{-1}), \frac{1}{\epsilon_{\Delta y}^Q}} \left( m \kappa^2(AS^{-1})\right)
\end{equation*}
queries to QRAM. The proof is complete.
\end{proof}

Now, we present the pseudocode of our inexact feasible quantum dual logarithmic barrier method in Algorithm~\ref{alg:qdlbm}, and provide the main theorem about the total complexity of Algorithm~\ref{alg:qdlbm}.
\begin{algorithm}[H]
\caption{IF-DQLBM}
\label{alg:qdlbm}
\begin{algorithmic}
\State \textbf{Initialize}: Choose $\zeta > 0$, $\theta = \frac{1}{4\sqrt{n}}$, $\tau \leq 0.5$, $(y^0,s^0) \in \mathcal{D}^\circ$, and $\mu^0 > 0$ such that $\delta(s^0,\mu^0)\leq \tau$.
\State $k \gets 0$
\While{$n\mu > \zeta$}
\State $(M^k, v^k) \gets$ \textbf{Build} system $M^k = AS^{-2}A^T, v^k = \frac{1}{\mu^k} r_p^k$
\State ${\Delta \bar{y}^k} \gets $ \textbf{Solve} the $(M^k, v^k)$ using QSVT and QTA
\State ${\Delta \bar{s}^k} = -A^T \Delta y^k$
\State $(y^{k+1},s^{k+1}) \gets (y^{k},s^{k}) + \lambda^{*k}({\Delta \bar{y}^k},{\Delta \bar{s}^k})$ according to Eq.~\eqref{eq: lambda ast}
\State $\mu \gets (1-\theta) \mu$
\State $k \gets k+1$
\EndWhile
\State \Return $(y^{k},s^{k})$
\end{algorithmic}
\end{algorithm}

\begin{theorem}\label{theorem: main original}
   Given $\zeta > 0$, $\theta = \frac{1}{4\sqrt{n}}$, proximity parameter $\tau \leq 0.5$, initial solution $(y^0,s^0) \in \mathcal{D}^\circ$, and $\mu^0 > 0$ such that $\delta(s^0,\mu^0)\leq \tau$, Algorithm~\ref{alg:qdlbm} finds a $\zeta$-approximate solution using 
    \begin{equation*}
        \widetilde{\Ocal}_{n,\|AS^{-1}\|_F,\kappa(AS^{-1}), \mu^0,  \frac{1}{\zeta}}\left( m \sqrt{n} \kappa^2(AS^{-1})
        \right)
    \end{equation*}
    queries to QRAM and 
    \begin{equation*}
        \Ocal\left( m {n}^{1.5} \log\left(\frac{n\mu^0}{\zeta}\right) \right)
    \end{equation*}
    classical arithmetic operations.
\end{theorem}

\begin{proof}
    We proved in Theorem~\ref{theorem: polynomial complexity} that our algorithm enjoys a $\Ocal \left(\sqrt{n} \log\left(\frac{n\mu^0}{\epsilon}\right)\right)$ iteration complexity. In each iteration, our subroutine solves the \eqref{eq:nes} system with complexity $\widetilde{\Ocal}_{n,\|AS^{-1}\|_F,\kappa(AS^{-1}), \frac{1}{\epsilon_{\Delta y}^Q}} \left( m \kappa^2(AS^{-1})\right)$, due to Theorem~\ref{Subroutine_Complexity}. On the classical side, each iteration incurs a cost of $\Ocal(mn)$ classical arithmetic operations.
\end{proof}

\subsection{Application of Iterative Refinement}
So far, we have developed and analyzed the inexact feasible dual logarithmic barrier algorithm that provides the solution $(y,s)$. 
Based on Theorem \ref{theorem: main original}, we have a polynomial dependence on $\kappa(AS^{-1})$, which is inherently dependent on $\kappa(A), \omega, \frac{1}{\zeta}$. 
The dependence of $\kappa(AS^{-1})$ on the inverse of precision translates into having exponential time complexity. To get an exact optimum, we need to get to precision $\zeta=2^{-\Ocal(L)}$, where $L$ is the binary length of input data defined as
\begin{equation*}
    L = m + n + mn + \sum_{i,j} \lceil\log(|a_{ij}|+1)\rceil + \sum_{i} \lceil\log(|c_{i}|+1)\rceil  + \sum_{j} \lceil\log(|b_{j}|+1)\rceil,  
\end{equation*}
and apply the rounding procedure.

To mitigate this dependence on precision, following the work of \cite{mohammadisiahroudi2023inexact}, we adopt the iterative refinement (IR) scheme. In \cite{mohammadisiahroudi2023inexact}, both primal and dual information are needed to run the iterative refinement scheme. Here, we propose an iterative refinement scheme that requires only dual information. It is worth mentioning that the idea of iterative refinement was first adopted for LO problems in the classical setting by \cite{gleixner2016iterative}. Also, similar approaches are also developed for semidefinite optimization \cite{mohammadisiahroudi2023quantum}.

For a feasible dual solution $(y,s)$ and a scaling factor $\nabla \geq 1$, we define the refining problem, at iteration $k$, as 
\begin{equation}\label{def: IR problem}\tag{$\Dcal_{\rm IR}$}
    \begin{aligned}
        \max_{\yhat,\shat} \ \nabla^{(k)} b^T \yhat\ \ & \\
        {\rm s.t. }\;\;
        A^T\yhat +&\shat = \nabla^{(k)} s^{(k)},\\
        &\shat \geq 0.
    \end{aligned}
\end{equation}
One can easily check that the dual of \eqref{def: IR problem} is as follows:
\begin{equation}\label{def: IR primal}\tag{$\mathcal{P}_{\rm IR}$}
    \begin{aligned}
        \min_{\xhat}\  \nabla^{(k)} {s^{(k)}}^T&\xhat, \\
        {\rm s.t. }\;\;
        A\xhat &= \nabla^{(k)} b, \\
        \xhat &\geq 0.
    \end{aligned}
\end{equation}
The associated proximity measure is defined as
\begin{equation*}
    \begin{aligned}
        \delta_{\rm IR}^{(k)}(\hat{s},\mu) = \frac{1}{\mu}\min_{\hat{x}} \{ \|\mu e - \hat{s}\hat{x}\|_2 : A\hat{x} = \nabla^{(k)}b \}.
    \end{aligned}
\end{equation*}
For the IR problem \eqref{def: IR problem}, we iteratively solve it using Algorithm \ref{alg:qdlbm} with precision $\hat{\zeta} = 10^{-2}$ until the complementarity gap reaches a desired accuracy $\zeta$. 
Once we solve the refining problem to precision $\hat{\zeta}$ at iteration $k$ to acquire $\yhat$, we update the current solution as 
\begin{align*}
    y^{(k+1)} &= y^{(k)} + \frac{1}{\nabla^{(k)}} \yhat, \\
    s^{(k+1)} &= c - A^{T}(y^{(k)} + \frac{1}{\nabla^{(k)}} \yhat) 
      =c-A^{T}y^{(k+1)}.
\end{align*}
We continue updating low-precision solutions until we reach a high-precision solution. The pseudocode of these steps is presented in Algorithm~\ref{alg:ir-if-qdlbm}.
\begin{algorithm}[H]
\caption{IR-IF-DQLBM}
\label{alg:ir-if-qdlbm}
\begin{algorithmic}
\State \textbf{Initialize}: Choose $0 < \zeta \ll \hat{\zeta}$, $(y^{(0)}, s^{(0)}) \in \mathcal{D}^\circ$, $\tau\leq0.5$, $\mu^{(0)} > 0$ such that $\delta(s^{(0)},\mu^{(0)})\leq \tau$, and $k=1$.
\State $(y^{(1)}, s^{(1)}) \gets$ Solve dual problem with $\hat{\zeta}$
\State $\nabla^{(0)} \gets 1$
\While{$\nabla^{(k-1)} < \frac{1}{\zeta}$}
\State $\nabla^{(k)} \gets \nabla^{(k-1)} \times \frac{1}{\hat{\zeta}}$
\State Construct the IR problem \eqref{def: IR problem}
\State $(\hat{y}, \hat{s}) \gets $ Solve \eqref{def: IR problem} using Algorithm~\ref{alg:qdlbm} to precision $\hat{\zeta}$
\State $y^{(k+1)} \gets y^{(k)} + \frac{1}{\nabla^{(k)}} \hat{y}$
\State $s^{(k+1)} \gets c - A^T y^{(k+1)}$
\State $k \gets k+1$
\EndWhile
\State \Return $(y^{(k)},s^{(k)})$

\end{algorithmic}
\end{algorithm}
Like any feasible IPM, as we solve the refining problem with a quantum interior point algorithm, we need a feasible initial solution for the refining problem in each iteration. The following lemma provides a feasible initial point for the refining problem at iteration $k$.
\begin{lemma}\label{lemma: IR feasibility}
    Given $(y^{(0)}, s^{(0)})$ an interior feasible solution for \eqref{LO: dual} and $(y^{(k)}, s^{(k)})$  the solution for \eqref{LO: dual} generated by the $(k-1)^{\rm th}$ iteration of Algorithm~\ref{alg:ir-if-qdlbm}. If $(y^{(k)}, s^{(k)})$ is interior feasible for \eqref{LO: dual}, then $(\nabla^{(k)} (y^{(0)} - y^{(k)}), \nabla^{(k)} s^{(0)})$ is an interior feasible solution for the refining problem \eqref{def: IR problem}. 
\end{lemma}
\begin{proof}
    First, we check the dual feasibility condition as follows.
    \begin{equation*}
        A^T (\nabla^{(k)} (y^{(0)} - y^{(k)})) + \nabla^{(k)} s^{(0)} = \nabla^{(k)} (A^T (y^{(0)} - y^{(k)}) + s^{(0)}) = \nabla^{(k)} (c- A^T y^{(k)}) = \nabla^{(k)} s^{(k)}.
    \end{equation*}
    Furthermore, it is clear that $\nabla^{(k)} s^{(0)} > 0$, hence, the solution is interior feasible for \eqref{def: IR problem}.
\end{proof}
The next lemma proves that the initial solution we choose for \eqref{def: IR problem} is close enough to its central path as required by Algorithm~\ref{alg:qdlbm}.
\begin{lemma}\label{lemma: IR initial delta}
In the $k^{\rm th}$ iteration of Algorithm~\ref{alg:ir-if-qdlbm}, let 
\begin{itemize}
    \item $(\nabla^{(k)}(y^{(0)}-y^{(k)}), \nabla^{(k)}s^{(0)})$ be the initial solution for \eqref{def: IR problem}, and
    \item $(\nabla^{(k)})^2\mu^{(0)}$ be the initial duality gap.
\end{itemize} 
Then, we have $\delta_{\rm IR}^{(k)}(\nabla^{(k)}s^{(0)}, (\nabla^{(k)})^2 \mu^{(0)})\leq 0.5$.
\end{lemma}
\begin{proof}
According to Algorithm~\ref{alg:ir-if-qdlbm}, we have $\delta(s^{(0)},\mu^{(0)})\leq 0.5$. Following Theorem~\ref{theorem: quadratic}, $x(s^{(0)},\mu^{(0)})$ is primal feasible and thus $Ax(s^{(0)},\mu^{(0)}) = b$ with $x(s^{(0)},\mu^{(0)})\geq 0$.
Furthermore, it is easy to see that
\begin{equation*}
    \begin{aligned}
        \arg \min_{\hat{x}}\{ \|(\nabla^{(k)})^2\mu^{(0)} e - (\nabla^{(k)}s^{(0)}) \hat{x}\|_2 : A\hat{x} = \nabla^{(k)}b \} &= \arg \min_{\hat{x}}\{ \|\mu^{(0)} e - s^{(0)} \hat{x}/(\nabla^{(k)})\|_2 : A\hat{x}/(\nabla^{(k)}) = b \} \\
        &= (\nabla^{(k)})\arg \min_{{x}}\{ \|\mu^{(0)} e - s^{(0)} {x}\|_2 : A{x} = b \}\\
        &= (\nabla^{(k)}) x(s^{(0)},\mu^{(0)}).
    \end{aligned}
\end{equation*}
By the definition of $\delta_{\rm IR}^{(k)}(\nabla^{(k)}s^{(0)}, (\nabla^{(k)})^2 \mu^{(0)})$, we have
\begin{equation*}
    \begin{aligned}
        \delta_{\rm IR}^{(k)}(\nabla^{(k)}s^{(0)}, (\nabla^{(k)})^2 \mu^{(0)}) 
        &= \frac{1}{\mu^{(0)}}\|\mu^{(0)}e - s^{(0)}x(s^{(0)},\mu^{(0)})\|_2\\
        &=\delta(s^{(0)},\mu^{(0)})\leq 0.5.
    \end{aligned}
\end{equation*}
\end{proof}
Moreover, we prove that in each iteration of Algorithm~\ref{alg:ir-if-qdlbm} there exists a primal solution associated with the computed dual solution such that the duality gap is bounded. The result is as follows. 
\begin{lemma}
In the $k^{\rm th}$ iteration of Algorithm~\ref{alg:ir-if-qdlbm}, the solution $(y^{(k+1)}, s^{(k+1)})$ is dual feasible, and the proximity measure of the solution satisfies
\begin{equation*}
    \begin{aligned}
        \delta(s^{(k)}, (\hat{\zeta})^{2k-1}) \leq 0.5.
    \end{aligned}
\end{equation*}
\end{lemma}
\begin{proof}
We prove this lemma by induction.
\begin{itemize}
    \item First, in the $1^{\rm st}$ iteration, we have a $\hat{\zeta}$-approximate interior feasible solution $(y^{(1)}, s^{(1)})$ for \eqref{LO: dual} and a $\hat{\zeta}$-approximate interior feasible solution $(\hat{y}^{(1)}, \hat{s}^{(1)})$ for \eqref{def: IR problem}. 
    According to the proof of Theorem~\ref{theorem: polynomial complexity}, we have $\delta_{\rm IR}^{(1)}(\hat{s}^{(1)},\hat{\zeta})\leq 0.5$. 
    Thus, there exists a feasible solution $\hat{x}^{(1)}(\hat{s}^{(1)},\hat{\zeta})$ for \eqref{def: IR primal}, i.e., $A\hat{x}^{(1)}(\hat{s}^{(1)},\hat{\zeta}) = b/\hat{\zeta}$, because of Theorem~\ref{theorem: quadratic}. By the end of the $1^{\rm st}$ iteration, we have
    \begin{equation*}
        \begin{aligned}
            y^{(2)} &= y^{(1)} + \hat{\zeta} \hat{y}^{(1)}\\
            s^{(2)} &= c- A^T y^{(2)} = s^{(1)} - \hat{\zeta} A^T\hat{y}^{(1)} = \hat{\zeta}\hat{s}^{(1)}> 0,
        \end{aligned}
    \end{equation*}
    which is strictly dual feasible for \eqref{LO: dual}.
    Then, we have
    \begin{equation*}
        \begin{aligned}
            \delta(s^{(2)}, (\hat{\zeta})^3) & = \frac{1}{(\hat{\zeta})^3} \min_{x} \{ \| (\hat{\zeta})^3 e - s^{(2)} x\|_2 : Ax=b  \}\\
            &\leq \frac{1}{(\hat{\zeta})^3} \| (\hat{\zeta})^3 e - s^{(2)}\hat{\zeta}\hat{x}^{(1)}(\hat{s}^{(1)},\hat{\zeta})\|_2\\
            &= \frac{1}{(\hat{\zeta})^3} \| (\hat{\zeta})^3 e - (\hat{\zeta})^2 \hat{s}^{(1)} \hat{x}^{(1)}(\hat{s}^{(1)},\hat{\zeta})\|_2\\
            &= \delta_{\rm IR}^{(1)}(\hat{s}^{(1)}, \hat{\zeta}) \leq 0.5.
        \end{aligned}
    \end{equation*}
    \item Second, by the end of the $(k-1)^{\rm th}$ iteration, let us assume that we have a strictly dual feasible solution $(y^{(k)}, s^{(k)})$ for \eqref{LO: dual} and $\delta(s^{(k)}, (\hat{\zeta})^{2k-1}) \leq 0.5.$

    \item Then, in the $k^{\rm th}$ iteration, if Algorithm~\ref{alg:ir-if-qdlbm} does not terminate, then we solve \eqref{def: IR problem} and obtain a $\hat{\zeta}$-approximate strictly feasible solution $(\hat{y}^{(k)}, \hat{s}^{(k)})$ with $\delta_{\rm IR}^{(k)}(\hat{s}^{(k)}, \hat{\zeta})\leq 0.5$. 
    According to Theorem~\ref{theorem: quadratic}, there exists a feasible solution $\hat{x}^{(k)}(\hat{s}^{(k)}, \hat{\zeta})$ for \eqref{def: IR primal}.
    Then, we construct the solution
    \begin{equation*}
        \begin{aligned}
            y^{(k+1)} &= y^{(k)} + (\hat{\zeta})^k \hat{y}^{(k)}\\
            s^{(k+1)} &= c- A^T y^{(k+1)} = s^{(k)} - (\hat{\zeta})^k A^T\hat{y}^{(k)} = (\hat{\zeta})^k\hat{s}^{(k)}> 0,
        \end{aligned}
    \end{equation*}
    which is a strictly dual feasible solution for \eqref{LO: dual}.
    It follows that
    \begin{equation*}
        \begin{aligned}
            \delta(s^{(k+1)}, (\hat{\zeta})^{2k+1}) &= \frac{1}{(\hat{\zeta})^{2k+1}} \min_{x} \{ \| (\hat{\zeta})^{2k+1} e - s^{(k+1)} x\|_2 :Ax=b\}\\
            &\leq \frac{1}{(\hat{\zeta})^{2k+1}} \| (\hat{\zeta})^{2k+1} e - s^{(k+1)} (\hat{\zeta})^k \hat{x}^{(k)}(\hat{s}^{(k)}, \hat{\zeta})\|_2\\
            &=\frac{1}{(\hat{\zeta})^{2k+1}} \| (\hat{\zeta})^{2k+1} e - (\hat{\zeta})^k\hat{s}^{(k)}(\hat{\zeta})^k \hat{x}^{(k)}(\hat{s}^{(k)}, \hat{\zeta})\|_2\\
            &= \delta_{\rm IR}^{(k)}(\hat{s}^{(k)}, \hat{\zeta}) \leq 0.5,
        \end{aligned}
    \end{equation*}
    which completes the proof.
\end{itemize}
\end{proof}

Given that in the context of iterative refinement, we terminate the subroutine early, i.e. $\mu \leq \hat{\zeta} = 10^{-2}$, the condition number of the system in that iteration remains bounded by a constant. In other words, the condition number of the system is bounded by a constant times the initial condition number, i.e. 
\begin{equation*}
    \kappa{(A(S^{(k)})^{-2}A^T)} = \Ocal\left(\frac{\kappa^0}{\mu^{(k)}}\right) = \Ocal(\kappa^0).
\end{equation*}
Thus, we can conclude that all NES systems in iterations of Algorithm \ref{alg:ir-if-qdlbm} have the $\Ocal(\kappa^0)$ uniform condition number bound \cite{mohammadisiahroudi2023quantum}. Thus, it remains to discuss the bound on $\kappa^0$. According to Lemma~\ref{lemma:cond} and the discussion thereafter, $\kappa^0$ depends on $\omega^0 = \|s^0\|_\infty, \mu^0$, and $\kappa_A$. Further, note that $\kappa^0 = \kappa_{(A(S^{(0)})^{-1})}^2$, by construction.

The following lemma bounds the number of iterations required to solve the problem to the desired precision $\zeta$.
\begin{lemma}\label{th:irinquiry}
Algorithm~\ref{alg:ir-if-qdlbm} produces a $\zeta$-optimal solution with at most $\Ocal(\frac{\log(\zeta)}{\log(\hat{\zeta})})$ inquiries to Algorithm~\ref{alg:qdlbm} with precision $\hat{\zeta}$. 
\end{lemma}
\begin{proof}
    According to Algorithm~\ref{alg:ir-if-qdlbm}, $\nabla^{(K)} = (\hat{\zeta})^{-K}$. If Algorithm~\ref{alg:ir-if-qdlbm} terminates after $K$ iterations, then 
    \begin{equation*}
        \begin{aligned}
            (\hat{\zeta})^{-K} \geq (\zeta)^{-1}  \Rightarrow K\geq \frac{\log(\zeta)}{\log(\hat{\zeta})}.
        \end{aligned}
    \end{equation*}
\end{proof}
Finally, by using iterative refinement, we obtain the complexity of Algorithm~\ref{alg:ir-if-qdlbm} as follows. 
\begin{theorem}
Algorithm~\ref{alg:ir-if-qdlbm} finds a $\zeta$-optimal solution using at most 
\begin{equation*}
        \widetilde{\Ocal}_{n, \|AS^{-1}\|_F, \kappa(AS^{-1}), \mu^0, \frac{1}{\zeta}}\left( m \sqrt{n} \kappa^0 
        \right)
    \end{equation*}    
    queries to QRAM and 
    \begin{equation*}
        \Ocal\left( m {n}^{1.5} \log\left(\frac{n\mu^0}{\zeta}\right) \right)
    \end{equation*}
    classical arithmetic operations.
\end{theorem}

Our work can be compared to the work of Apers and Gribling~\cite{apers2023quantum} as we both present quantum interior point methods to solve linear optimization problems in dual form, assuming having access to QRAM. As shown in Table~\ref{tab:compelxities}, in~\cite{apers2023quantum}, the proposed algorithm makes $\widetilde{\Ocal}_{m,n, \frac{1}{\zeta}}(mn), \widetilde{\Ocal}_{m,n, \frac{1}{\zeta}}(m^{5/4} n^{3/4}), \widetilde{\Ocal}_{m,n, \frac{1}{\zeta}}(m^5 \sqrt{n})$ queries to the data depending on which of the logarithmic, volumetric, or Lewis weight barrier is used, respectively. It is worth mentioning that their algorithm is tailored to specific types of problems. These sorts of problems appear in different applications such as linear relaxation of discrete optimization problems, machine learning problems, Chebyshev approximation, and linear separability problems. However, to get a sublinear number of queries the structure of the coefficient matrix should be a tall-and-skinny sparse matrix, i.e., $n \geq \gamma m^{10}$, where $\gamma$ is some constant. In contrast, our algorithm does not require such a special structure and can provide a sublinear number of queries for problems with coefficient matrix having $n \geq \vartheta m^2$, for some constant $\vartheta$. Therefore, in worst-case complexity analysis, the algorithm of \cite{apers2023quantum} has a $\Ocal(n^{7.5})$ complexity which is worse than ours and other works. 

Further, we can compare our work with primal-dual IPMs (PD-IPMs). If we assume $n \geq \vartheta m^2$ to get a sublinear number of queries, we subsequently get $\Ocal(n^2)$ number of classical arithmetic operations, which is better than that of PD-IPMs ($\Ocal(n^{2.5})$). Further, PD-IPMs require preprocessing or modification of the Newton system to ensure that the iterates remain feasible \cite{mohammadisiahroudi2023inexact}. Our framework does not need any preprocessing or modification of the system. Further, the quantum complexity of the majority of quantum PD-IPMs depends on the norm of the optimal solution ($\omega$) \cite{mohammadisiahroudi2024efficient,mohammadisiahroudi2023inexact, mohammadisiahroudi2023improvements}. This term can be extremely large leading to a high quantum complexity in those works. To mitigate this dependence, preconditioning is needed which is used in the work of \cite{mohammadisiahroudi2023improvements}. In contrast, our algorithm does not need any preconditioning as our complexity is not dependent on the norm of the optimal solution. Instead, as stated earlier, our algorithm complexity depends on the condition number of the initial Newton system. The initial system condition number depends on the condition number of the input data and the norm of the initial solution which can be bounded moderately. These are some of the advantages that using a dual-only framework can yield. A comparison of our work with existing works is presented in Table~\ref{tab:compelxities}.

\begin{table}[H]
\centering
\renewcommand{\arraystretch}{1.5}
\resizebox{\textwidth}{!}{%
\begin{tabular}{cccc}
\hline
Algorithm                                                  & Linear System Solver   & Quantum Complexity                                                           & Classical Complexity                                                                  \\ \hline
Dual logarithmic barrier method~\cite{roos1997theory}      & Cholesky               &                                                                              & $\tilde{\Ocal}_{\frac{1}{\zeta}}(m^{2} n^{1.5})$                                      \\ \hline
PD IPM with Partial Updates~\cite{roos1997theory}          & Low rank updates       &                                                                              & $\tilde{\Ocal}_{\frac{1}{\zeta}}(n^{3})$                                              \\ \hline
PD Feasible IPM~\cite{roos1997theory}                      & Cholesky               &                                                                              & $\tilde{\Ocal}_{\frac{1}{\zeta}}(n^{3.5})$                                            \\ \hline
Robust IPM~\cite{van2020solving}                           & Spectral Approximation &                                                                              & $\tilde{\Ocal}_{\frac{1}{\zeta}}(n^{\omega_0})$                                       \\ \hline
PD II-IPM~\cite{monteiro2003convergence}                   & PCGM                   &                                                                              & $\tilde{\Ocal}_{\frac{1}{\zeta}}(n^{5}\bar{\chi}^2)$                                  \\ \hline
PD II-QIPM~\cite{mohammadisiahroudi2024efficient}          & QLSA+QTA               & $\tilde{\Ocal}_{n,\kappa_{A}, \omega, \frac{1}{\zeta}}(n^{4}\kappa_A^{4}\omega^{19}\|A\|)$   & $\tilde{\Ocal}_{\omega, \frac{1}{\zeta}}(n^{4} )$                     \\ \hline
PD IF-QIPM~\cite{mohammadisiahroudi2023inexact}            & QLSA+QTA               & $\tilde{\Ocal}_{n,\kappa_{A}, \omega, \frac{1}{\zeta}}(n^{2}\kappa_A^{2}\omega^{5} \|A\|)$   & $\tilde{\Ocal}_{\mu^0, \frac{1}{\zeta}}(n^{2.5})$                     \\ \hline
PD IR-IF-IPM~\cite{mohammadisiahroudi2023improvements}     & PCGM                   &                                                                              & $\tilde{\Ocal}_{\mu^0, \frac{1}{\zeta}}(n^{3.5}\bar{\chi}^2 )$                        \\ \hline
PD IR-IF-QIPM~\cite{mohammadisiahroudi2023improvements}    & QLSA+QTA               & $ \widetilde{\Ocal}_{ n,\kappa_{\Ahat}, \|\Ahat\|,\|\bhat\|,\mu^0,\frac{1}{\zeta}} ( m\sqrt{n} \bar{\chi}^2 \omega^2 )$ & $\tilde{\Ocal}_{\mu^0, \frac{1}{\zeta}}(n^{2.5})$   \\ \hline
IPM with approximate Newton steps~\cite{apers2023quantum}  & Q-spectral Approx.     & $\tilde{\Ocal}_{m,n, \frac{1}{\zeta}}(m^5 \sqrt{n})$                                 & $\tilde{\Ocal}_{\frac{1}{\zeta}}(n^{1.5} )$                                   \\ \hline
\textbf{This work}                                         & QLSA+QTA               & $\widetilde{\Ocal}_{n, \|AS^{-1}\|_F, \kappa(AS^{-1}), \mu^0,\frac{1}{\zeta}}\left(m \sqrt{n} \kappa^0 \right)$ & $\tilde{\Ocal}_{\frac{1}{\zeta}}(m {n}^{1.5})$                  \\ \hline
\end{tabular}%
}
\caption{\small Complexity of different IPMs for LO (Quantum complexity is expressed in the form of query complexity. Here, $\omega$ is the upper bound on the norm of the optimal solution, and $\bar{\chi}^2$ is an upper bound on the condition number of the Newton system). PD stands for Primal-Dual, and PCGM stands for Preconditioned Conjugate Gradient Method.}
\label{tab:compelxities}
\end{table}

To give a better comparison of our work, if we compare the complexity of our algorithm under the condition where it yields a sublinear number of queries, i.e. $n \geq \Ocal(m^2)$, we have the following complexities among the existing classical and quantum algorithms as shown in Table \ref{tab:compelxities-simple}.

\begin{table}[H]
\centering
\renewcommand{\arraystretch}{1.5}
\resizebox{\textwidth}{!}{%
\begin{tabular}{cccc}
\hline
Algorithm                                                  & Linear System Solver   & Quantum Complexity                                                           & Classical Complexity                                \\ \hline
Dual logarithmic barrier method~\cite{roos1997theory}      & Cholesky               &                                                                              & $\tilde{\Ocal}_{\frac{1}{\zeta}}(n^{2.5})$                                      \\ \hline
IPM with Partial Updates~\cite{roos1997theory}             & Low rank updates       &                                                                              & $\tilde{\Ocal}_{\frac{1}{\zeta}}(n^{3})$                                     \\ \hline
PD Feasible IPM~\cite{roos1997theory}                      & Cholesky               &                                                                              & $\tilde{\Ocal}_{\frac{1}{\zeta}}(n^{3.5})$                                   \\ \hline
Robust IPM~\cite{van2020solving}                           & Spectral Approximation &                                                                            & $\tilde{\Ocal}_{\frac{1}{\zeta}} \left( n^{\omega_0} \right)$ \\ \hline
PD IR-IF-QIPM~\cite{mohammadisiahroudi2023improvements}    & QLSA+QTA               & $\widetilde{\Ocal}_{ n,\kappa_{\Ahat}, \|\Ahat\|,\|\bhat\|,\mu^0, \frac{1}{\zeta}} ( n \bar{\chi}^2 \omega^2 )$ & $\tilde{\Ocal}_{\mu^0, \frac{1}{\zeta}}(n^{2.5})$ \\ \hline
IPM with approximate Newton steps~\cite{apers2023quantum}  & Q-spectral Approx.     & $\tilde{\Ocal}_{m,n, \frac{1}{\zeta}}(n^3)$                                  & $\tilde{\Ocal}_{\frac{1}{\zeta}}(n^{1.5} )$                                  \\ \hline
\textbf{This work}                                         & QLSA+QTA               & $\widetilde{\Ocal}_{n, \|AS^{-1}\|_F, \kappa(AS^{-1}), \mu^0, \frac{1}{\zeta}}\left( n  \kappa^0 \right)$ & $\tilde{\Ocal}_{\frac{1}{\zeta}} \left( {n}^{2}  \right) $                    \\ \hline
\end{tabular}%
}
\caption{\small Complexity of different IPMs for LO considering $n \geq \Ocal(m^2)$}
\label{tab:compelxities-simple}
\end{table}

\section{Conclusion} \label{sec:concl}
This paper proposes an algorithm for solving linear optimization problems using the dual logarithmic barrier function. In the convergence analysis, we showed that our algorithm has quadratic convergence toward the central path despite the inexact directions. Moreover, our algorithm enjoys an $\Ocal(\sqrt{n})$ IPM iteration complexity. Further, due to the promising speed-up that quantum linear system solvers provide for IPMs, we use this type of algorithm to speed up the most expensive part of the algorithm, solving the Newton system to acquire the search directions. We exploited QSVT of \cite{gilyen2019qsvt} and QTA of \cite{van2023quantum} as a subroutine for solving the normal equation system in each iteration. By applying iterative refinement, we achieve a complexity of $\tilde{\Ocal}(m\sqrt{n}  \kappa^0)$ number of queries to QRAM and $\tilde{\Ocal}(mn^{1.5} L)$ number of classical arithmetic operations. In comparison to the existing works, our work requires milder conditions to achieve a sublinear number of queries to that of the recent similar work of \cite{apers2023quantum}. Further, in comparison to primal-dual IPMs, our work has a better dependence with respect to the condition number and can yield a cheaper classical arithmetic cost under the assumption of $n \geq \vartheta m^2$ on the dimension of the problem. We highlight the fact that iterative refinement plays a major role in mitigating the impact of the condition number, by eliminating the precision dependence. Nevertheless, our complexity still has a dependence on the condition number of the initial Newton system which is dependent on the norm of the initial solution and the condition number of the input data. Efforts to initialize the algorithm by a solution with a moderately bounded norm in addition to preconditioning the input data matrix create an avenue of research to further improve the complexity of our algorithm.

\section*{Acknowledgements}
This work is supported by the Defense Advanced Research Projects Agency as part of the project W911NF2010022: {\em The Quantum
Computing Revolution and Optimization: Challenges and Opportunities}; and by the National Science Foundation (NSF) under Grant No. 2128527.
\bibliographystyle{abbrvnat} 
\bibliography{bib}

\end{document}